\documentclass[11pt, final]{amsart}
\usepackage{graphicx}
\usepackage{latexsym}
\usepackage{amsfonts,amsmath,amssymb}
\usepackage{url}
\usepackage{tikz}
\usepackage{float}
\usepackage[indent]{parskip}
\usepackage{xcolor}
\usepackage{ifdraft}
\ifoptionfinal{}{\pagecolor[HTML]{FDF6E3}}

\newtheorem{theorem}{Theorem}[section]								
\newtheorem{lemma}[theorem]{Lemma}
\newtheorem{definition}[theorem]{Definition}
\newtheorem{corollary}[theorem]{Corollary}

\newtheorem{conjecture}[theorem]{Conjecture}

\usepackage{color}

\def\part{\partial}

\def\b1{\bold 1}

\newcommand{\beq}{\begin{equation}}
\newcommand{\eeq}{\end{equation}}

\theoremstyle{remark}

\numberwithin{equation}{section}
\date{\today}

\begin{document}

\title[Fixed number of occurrences of a pattern]{Permutations with a fixed number of occurrences of a pattern: a case generalizing 231}
\author {Michael Waite}
\address{Department of Mathematics, University of Florida, Gainesville, FL 32601}
\email{michael.waite@ufl.edu}

\begin{abstract} 
We determine a set of permutation patterns $q$ so that the number of permutations with $r$ occurrences of $q$ is asymptotically $n^r$ times the number of permutations avoiding $q$, partially settling a conjecture of Conway and Guttmann. We also use these asymptotics to prove nonrationality and nonalgebraicity for certain ordinary generating functions for permutations with $r$ copies of a pattern.
\end{abstract}

\maketitle

\section{Introduction}
We say that a permutation $p$ contains a pattern $q = q_1 ... q_k$ if there is a subsequence $n_1, ..., n_k$ so that $p_{n_i} < p_{n_j}$ if and only if $q_i < q_j$. We say that $p$ contains $r$ copies of $q$ if there are $r$ different such subsequences $n_1, ..., n_k$.

We denote the set of permutations of length $n$ that avoid $q$ as $S_n(q)$, and we denote the set of permutations with $r$ copies of $q$ as $S_{n,r}(q)$.

It is well-known that $|S_n(321)| = |S_n(231)|  = \frac{1}{n+1} {2n \choose n}$, which is the $n$th Catalan number. However, it is not true that $|S_{n,r}(321)| = |S_{n,r}(231)|$ for $r\geq 1$. It is known due to Noonan \cite{Noonan} that $|S_{n,1}(321)| = \frac{3}{n} { 2n \choose n+3 }$ and it is known due to Bóna \cite{Bona1998} that $|S_{n,1}(231)| = {2n-3 \choose n-3 } $.

Enumerations of $|S_{n,r}(q)|$ have been done for small values of $r$ for the length $3$ patterns, but a perhaps more interesting question is what happens as we vary $r$. It is known due to Bóna \cite{Bona1997a} that there is some constant $C$ so that $|S_{n,r}(231)| \sim C \frac{4^n} {n^{3/2 - r}}$, and the present author \cite{Waite2025} has recently shown that $|S_{n,r}(321)| = \Theta( \frac{4^n} {n^{3/2}})$, and moreover that for any permutation $q$ it is true that $|S_{n,r}(321 \ominus q)| = \Theta (|S_{n}(321 \ominus q)|)$.

Recently, Conway and Guttmann \cite{Conway} have found substantial numerical evidence for the asymptotics of $|S_{n,r}(q)|$ where $q$ is of length $4$. For five of the seven different classes, they have conjectured that $|S_{n,r}(q)| \sim C n^r |S_{n}(q)|$

In this paper, we will find a set of patterns so that there is a constant $C$ so that $|S_{n,r}(q)| = \Theta(n^r |S_{n}(q)|)$, as is known to be true for $q = 231$.

In Section 2, we will construct an injection that gives us the desired upper bound, and in Section $3$, we will construct an injection that gives us the desired lower bound. In Section $4$, we state the main result of the paper, combining the necessary conditions on the pattern $q$ from the upper bound and from the lower bound. In Section $5$ we will apply these results for patterns where the asymptotics are known, and obtain new results on the nonrationality and nonalgebraicity of the generating functions for permutations with $r$ copies of these patterns. Finally, in Section $6$ we return to the conjecture of Conway and Guttmann \cite{Conway} and discuss which cases are remaining.

\section{A Lower Bound}

First, we need a definition.
\begin{definition}
    Suppose $n\geq k$. Let $\mathcal{I}_{n,k}$ denote the set of size $k$ subsets of the set $[n] = \{1,2, ..., n \}$.
\end{definition}
Now, we will state and prove the main result of this section.
\begin{theorem}
    Let $q$ be a permutation of length $m$ which is indecomposable, and is of the form $q_1 \ominus ... \ominus q_k$ for some $k \geq 2$ where all of the $q_i$ are skew-indecomposable and $q_2, ..., q_{k-1}$ are not $1$. Let $q_1$ be of length $\ell$.  Suppose $r \geq 1$. Then there is an injection
    \[
    \phi: \mathcal{I}_{n-rm-\ell+1, r} \times S_{n - rm - \ell}(q) \hookrightarrow S_{n,r}(q)
    \]
    for $n\geq rm$.

\end{theorem}

\begin{proof}
    Let $p\in S_{n - rm - \ell }(q)$ and let $S \in \mathcal{I}_{n-rm-\ell+1, r}$. Let $s_1, ..., s_{r}$ be the entries in $S$ in increasing order. Replace $p$ by $q_1 \oplus p \in Av_{n - rm}(q)$, so that we can assume that $p$ begins with a copy of $q_1$, and replace each $s_i$ by $s_i + \ell$.
    
    Our main idea is that we will insert a copy of $q$ in certain blocks of consecutive values at each of the positions $s_1, ..., s_{r}$ in $p$, and we will choose these blocks of values in such a way as to not create extra $q$ patterns.

    Now we will define these blocks of values. Let $B_1$ consist of $q_1$ and all of the values of $q_2$ except for the smallest value. For $2\leq i \leq k-2$, let $B_i$ be the set of values consisting of the smallest value of $q_i$ and all of the values of $q_{i+1}$ except for the smallest value. Let $B_{k-1}$ be the set of values consisting of the smallest value of $q_{k-1}$ and all of $q_k$. For $1 \leq i \leq k$ let $b_i = |B_i|$. 
    

    Now we will insert copies of $q$ iteratively starting at $s_1$. For $1\leq i \leq k-1$ let $a_i$ be the largest entry that is the smallest entry of a $q_1 \ominus ... \ominus q_i$ pattern that is entirely to the left of $s_1$. Insert a copy of $q$ in position $s_1$ with the values of $B_i$ in between $a_i$ and the next larger entry.

    Let $L_1$ be the subsequence of entries to the left of $s_1$ with value greater than $a_1$. For $2\leq i \leq k-1$, let $L_i$ be the subsequence of entries to the left of $s_1$ with value greater than $a_i$, and less than or equal to $a_{i-1}$. Let $L_{k}$ be the subsequence of entries to the left of $s_1$ with value less than or equal to $a_{k-1}$. Likewise define $R_1$, ..., $R_k$ for entries to the right of $s_1$.

    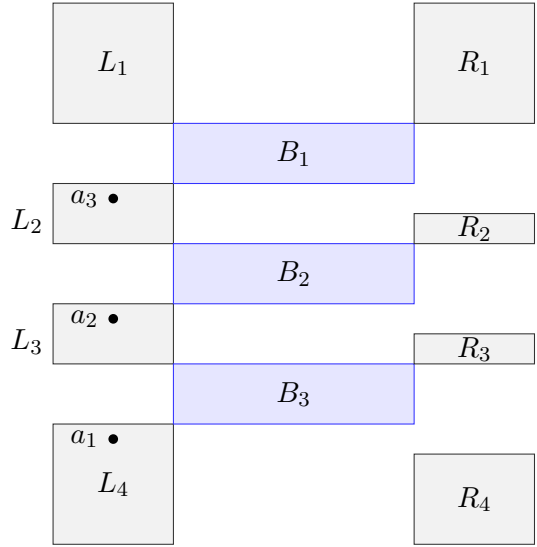
\begin{figure}[H]
        \centering
        \begin{tikzpicture}[scale = .8, point/.style={circle,draw=black!100,fill=black!100,thick,
                     inner sep=0pt,minimum size=1mm},
                          ]
                          
      \draw[draw = black!80!,fill=gray!10] (0,0) rectangle node[anchor= center]{$L_4$} (2,2);
      \draw[draw = black!80!,fill=gray!10] (0,3) node[anchor= south east]{$L_3$} rectangle  (2,4);
      \draw[draw = black!80!,fill=gray!10] (0,5) node[anchor= south east]{$L_2$}rectangle  (2,6);
      \draw[draw = black!80!,fill=gray!10] (0,7) rectangle node[anchor= center]{$L_1$} (2,9);

      \draw[draw = blue!80!,fill=blue!10] (2,2) rectangle node[anchor= center]{$B_3$} (6,3);
      \draw[draw = blue!80!,fill=blue!10] (2,4) rectangle node[anchor= center]{$B_2$} (6,5);
      \draw[draw = blue!80!,fill=blue!10] (2,6) rectangle node[anchor= center]{$B_1$} (6,7);

      \draw[draw = black!80!,fill=gray!10] (6,0) rectangle node[anchor= center]{$R_4$} (8,1.5);
      \draw[draw = black!80!,fill=gray!10] (6,3) rectangle node[anchor= center]{$R_3$} (8,3.5);
      \draw[draw = black!80!,fill=gray!10] (6,5) rectangle node[anchor= center]{$R_2$} (8,5.5);
      \draw[draw = black!80!,fill=gray!10] (6,7) rectangle node[anchor= center]{$R_1$} (8,9) ;

      \node (a1) at (1,1.75) [point] [label=left:$a_1$] {};
      \node (a2) at (1,3.75) [point] [label=left:$a_2$] {};
      \node (a3) at (1,5.75) [point] [label=left:$a_3$] {};

        \end{tikzpicture}
        \caption{Insertion of a $q$ pattern at position $s_1$, where $k = 4$}.
        \label{fig:placeholder}
    \end{figure}

    For $1\leq i \leq k-1$, by the definition of $a_i$ then the entries contained in $L_1 \cup ... \cup L_i$ cannot contain a $q_1 \ominus ... \ominus q_i$. Because $p$ does not contain any $q$ patterns with entries to the right of $s_1$, then for $2 \leq i \leq k$ the entries contained in $R_i \cup ... \cup R_k$ cannot contain a $q_i \ominus ... \ominus q_k$ pattern.

    Say there were a $q$ pattern with some entries in the $B_i$ blocks and some entries not in the $B_i$ blocks. Let $m$ be minimum so that there are some entries in this $q$ pattern in $B_m$, and $M$ be maximum so that there are some entries in this $q$ pattern in $B_M$. 

    Since $L_1 \cup ... \cup L_m$ cannot contain $q_1 \ominus ... \ominus q_m$, then the $q_m$ in the $q$ pattern must have some entries in $B_m$. Since $R_{M+1} \cup ... \cup R_k$ cannot contain any $q_{M+1} \ominus ... \ominus q_k$ pattern, then the $q_{M+1}$ in the $q$ pattern must have some entries in $B_M$. The possible ways to place a $q_i$ for $m\leq i \leq M+1$ so that not all of the entries are in a $B_j$ for some $j$ now must be either with some entries to the left of $B_m$ and some entries in $B_m$, some entries in $B_m$ and some entries in $B_{m+1}$, ..., some entries in $B_{M-1}$ and some entries in $B_M$, or some entries in $B_M$ and some entries to the right of $B_M$, or if $m = 1$ then all of $q_m$ may be entirely in $B_m$ or if $M = k$ then all of $q_M$ may be entirely in $B_M$. There are $M-m + 2$ of these spots and $M-m + 1$ blocks $q_m, ... q_{M+1}$, so it must be the case that either $q_m$ has some entries to the left of $B_m$ or $q_{M+1}$ has some entries to the right of $B_M$, or that $m=1$ and $M=k$ and all of $q_m$ is in $B_m$ and all of $q_M$ is in $B_M$. The latter case can be discarded since in this case $q$ will have no entries outside of the $B_i$ blocks which is a contradiction. Specifying to the former case, since $q_m$ and $q_{M+1}$ are skew indecomposable then either $q_m$ has some entries below and to the left of $B_m$ or $q_{M+1}$ has some entries above and to the right of $B_M$ .
    
    If $q_{M+1}$ has some entries above and to the right of $B_M$, then all of $q_1 ... q_M$ must be in $B_m \cup ... \cup B_{M-1} \cup L_1 \cup ...\cup L_m$. Since $L_1 \cup ...\cup  L_m$ cannot contain $q_1 \ominus ... \ominus q_m$, then $q_m$ must have some entries in $B_m$. So $B_m \cup ... \cup B_{M-1}$ must contain some entries of $q_m$ and $q_{M+1}$ and all of the entries of $q_{m+1}, ..., q_M$. Since there are $M-m$ blocks $B_m ... B_{M-1}$ and $M-m$ blocks of entries $q_{m+1}, ... q_M$, then it must be that some $q_i$ is entirely contained in some $B_j$, which is a contradiction.

    The case where $q_m$ has some entries below and to the left of $B_m$ results in a similar contradiction.

    
    Notice now that in the above argument we only required that the entries after $s_1$ do not contain entries in a $q$ pattern. This condition is now satisfied for $s_2$, so we can repeat the above argument for $s_2$, and again for $s_3, ..., s_{r}$, until we get a permutation in $S_{n,r}(q)$ which we call $\phi(p)$.

    To see that $\phi$ is injective, given $\phi(p)$ we can remove all of the entries that participate in $q$ patterns except for the leftmost $q$ pattern, and store the $r$ positions where each consecutive $q$ pattern was inserted to obtain $S$. Finally we can remove the leftmost $q_1$ pattern from the permutation to obtain $p$.
\end{proof}
Now, we note that this injection gives us exactly the desired lower bound.
\begin{corollary}
    Let $q$ be a permutation of length $m$ which is indecomposable, and is of the form $q_1 \ominus ... \ominus q_k$ for some $k \geq 2$ where all of the $q_i$ are skew-indecomposable and $q_2, ..., q_{k-1}$ are not $1$. Suppose $r \geq 1$. Then
    \[
    |S_{n,r}(q)| \geq kn^r|S_{n}(q)|
    \]
    for some constant $k$.
\end{corollary}
\begin{proof}
    This follows immediately from the previous Theorem.
\end{proof}

\section{An Upper Bound}
We will need the following definition.
\begin{definition}
    Let $S^*_{n,r}(q)$ be the set of permutations of length $n$ which contain exactly $r$ copies of $q$ and are such that none of these $r$ copies of $q$ have any two entries in common.
\end{definition}
We are now going to construct an injection on $S^*_{n,r}(q)$ sets. We will do this inductively, with the following lemma being the base case of this induction.
\begin{lemma}
    There is an injection 
    \[
    \psi: S_{n,r}^*(21) \hookrightarrow [n] \times S^*_{n,r-1}(21)
    \]
    which takes a permutation $p \in S_{n,r}^*(21)$ and swaps two entries, and does not create any new $21$ patterns which were not in $p$.
\end{lemma}

\begin{proof}
    Let $p \in S_{n,r}^*(21)$ and let $p_ip_j$ be the $21$ pattern in $p$ chosen so that $p_i$ is leftmost. We will set $\psi(p) = (\psi_1(p), \psi_2(p))$ where $\psi_1(p) = i$ and $\psi_2(p)$ is the permutation obtained by swapping $p_i$ and $p_j$ in $p$.

    Note that since no two copies of $21$ in $p$ intersect, then the only places where there may be entries in $p$ are to the left of $i$ with value below $p_j$, and to the right of $j$ with value above $p_i$.

    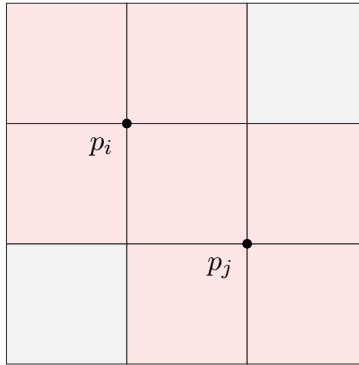
\begin{figure}[H]
        \centering
        \begin{tikzpicture}[scale = .8, point/.style={circle,draw=black!100,fill=black!100,thick,
                     inner sep=0pt,minimum size=1mm},
                          ]
                          
      \draw[draw = black!80!,fill=gray!10] (0,0) rectangle node[anchor= center]{} (2,2);
      \draw[draw = black!80!,fill=red!10] (2,0) rectangle node[anchor= center]{} (4,2);
      \draw[draw = black!80!,fill=red!10] (4,0) rectangle node[anchor= center]{} (6,2);
      \draw[draw = black!80!,fill=red!10] (0,2) rectangle node[anchor= center]{} (2,4);
      \draw[draw = black!80!,fill=red!10] (2,2) rectangle node[anchor= center]{} (4,4);
      \draw[draw = black!80!,fill=red!10] (4,2) rectangle node[anchor= center]{} (6,4);
      \draw[draw = black!80!,fill=red!10] (0,4) rectangle node[anchor= center]{} (2,6);
      \draw[draw = black!80!,fill=red!10] (2,4) rectangle node[anchor= center]{} (4,6);
      \draw[draw = black!80!,fill=gray!10] (4,4) rectangle node[anchor= center]{} (6,6);

      \node (pi) at (2,4) [point] [label=below left:$p_i$] {};
      \node (pj) at (4,2) [point] [label=below left:$p_j$] {};

        \end{tikzpicture}
        \caption{Diagram of $p_i$ and $p_j$ in $p$. The red squares contain no entries and the gray squares may contain entries.}
        \label{fig:placeholder}
    \end{figure}
    As such, $p_i$ and $p_j$ have consecutive positions and consecutive values. Thus swapping $p_i$ and $p_j$ cannot create any new $21$ patterns, and removes the $21$ pattern given by $p_ip_j$, thus $\psi_2(p)$ has exactly $r-1$ copies of $21$ and none of the copies share any two entries.

    To see that this map is injective, given $\psi(p)$ we may simply swap the positions $\psi_1(p)$ and $\psi_1(p) + 1$ in $\psi_2(p)$ to obtain $p$.
\end{proof}
Now we will inductively construct the full injection. 
\begin{theorem}
    Suppose $q$ is a separable skew-decomposable permutation of size $k\geq 2$. Then there is an injection
    \[
    \psi: S_{n,r}^*(q) \hookrightarrow [n] \times S^*_{n, r-1}(q)
    \]
    which takes a permutation $p \in S_{n,r}^*(q)$ and swaps two entries, and does not create any new $q$ patterns which were not in $p$.
\end{theorem}

\begin{proof}
    We will prove this by induction on $k$. The case $k=2$ is proven in the previous lemma. 
    
    So assume $k \geq 3$. Take the complement of $q$ and write $q = q_1 \oplus q_2$ for some $q_1$ and $q_2$. We can assume that $q_2$ has size at least $2$ (otherwise take the reverse complement). Since $q$ is separable, then $q_2$ is either skew-decomposable or decomposable.

    Suppose first that $q_2$ is skew-decomposable. Then there is an injection 
    \[
    \psi_{q_2}: S_{n,r}^*(q_2) \hookrightarrow [n] \times S^*_{n, r-1}(q_2).
    \]
    Let $S$ be the set of entries of $p\in S^*_{n,r}(q)$ which sit above and to the right of $q_1$. To define $\psi$, simply apply $\psi_{q_2}$ to $S$ to obtain the permutation and also return the position in $p$ of the entry being decreased in the swap. To show that $\psi$ is injective, it suffices to show that it is possible to determine $S$ given $\psi(p)$. 
    
    Now, in $\psi(p)$ the entry that is increased by the swap will continue to be above and to the right of $q_1$. Any entry which was not in $S$ which ends up above and to the right of $q_1$ must only be above and to the right of copies of $q_1$ that include the entry that is decreased in the swap. No entry which begins in $S$ can leave $S$ except for possibly the smaller swapped value, since the entry would need to be above and to the right of one of the swapped entries for the swap to affect it, but the swapped entries were in $S$ to begin with so there is some further copy of $q_1$ below and to the left which is unaffected. Since $\psi(p)$ gives the position of the smaller swapped value, then $S$ can be determined entirely from $\psi(p)$, and thus $\psi$ is injective.

    Suppose now that $q_2$ is decomposable. By taking complements, there is an injection
    \[
    \psi'_{q_2}: S_{n,r}(q_2) \hookrightarrow [n] \times S_{n, r-1}(q_2)
    \]
    with the first coordinate giving the position of the entry being \textit{increased}. Let $S$ be the set of entries of $p\in S_{n,r}(q)$ which sit above and to the right of $q_1$. To define $\psi$, simply apply $\psi_2$ to $S$ to obtain the permutation and return the position of the entry being increased. To show that $\psi$ is injective, it suffices to show that is is possible to determine $S$ given $\psi(p)$. 
    
    Now, in $\psi(p)$ the entry that is decreased by the swap will continue to be above and to the right of $q_1$ since the smaller value is being moved from the left and was in $S$ to begin with. No entry that was not in $S$ may end up above and to the right of $q_1$ since such an entry must sit above and to the right of one of the swapped entries, but then must have been above and to the right one of the entries before the swap and thus must have been above and to the right of $q_1$. No entry that begins in $S$ may end up not in $S$ since any such entry must be in $\psi(p)$ somewhere above and to the right of the position of the smaller swapped entry in $p$, and hence above and to the right of $q_1$. Hence $S$ is unchanged by $\psi(p)$, and thus $\psi$ is injective.
\end{proof}
From this injection, we can obtain the desired upper bound on not just the $S^*_{n,r}(q)$ sets but actually the $S_{n,r}(q)$ sets themselves.
\begin{corollary}
    Suppose $q$ is a separable skew-decomposable permutation of size $k\geq 2$ and let $r\geq 1$. Then
    \[
    |S_{n,r}(q)| \leq Kn^r |S_{n}(q)|
    \]
    for some constant $K$.
\end{corollary}

\begin{proof}
    We will prove the statement by induction on $r$. It is true that $S_{n,1}(q) = S_{n,1}^*(q)$ and $S_{n,0}(q) = S_{n,0}^*(q)$, so the previous theorem proves the result for $r = 1$.

    Now suppose $r\geq 2$. By applying the lemma twice we have
    \[
    |S_{n,r}^*(q)| \leq n^2 |S_{n,r-2}^*(q)|.
    \]
    Let $S_{n,r}^{\text{int}}(q) = S_{n,r}(q) \backslash S_{n,r}^*(q)$ be the permutations with $r$ copies of $q$ where there are two copies of $q$ which share some entry. There is an injection
    \[
    \phi: S_{n,r}^{\text{int}}(q) \hookrightarrow [n]^2 \times \cup_{i=0}^{r-2} S_{n-1,i}(q)
    \]
    which is given by taking $p\in S_{n,r}^{\text{int}}(q)$ and taking the first entry which is in at least two copies of $q$, say $p_i$, and setting $\phi(p) = (i,p_i, \phi_3(p))$ where $\phi_3(p)$ is the permutation obtained by removing $p_i$ from $p$. As such
    \[
    |S_{n,r}^{\text{int}}(q)| \leq n^2 \sum_{i=0}^{r-2} |S_{n-1,i}(q)| \leq K_1 n^2 |S_{n,r-2}(q)|
    \]
    for some constant $K_1$.

    Combining the two statements we get
    \[
    |S_{n,r}(q)| \leq (K_1 + 1) n^2 |S_{n,r-2}(q)|.
    \]
    Now by induction we have $|S_{n,r-2}(q)| \leq K_2 n^{r-2} |S_n(q)|$ and therefore
    \[
    |S_{n,r}(q)| \leq (K_1 + 1) K_2 n^r |S_{n}(q)|
    \]
    which completes the proof.
\end{proof}

\section{Main Result}
Now we may state the main result.
\begin{theorem}
    Let $q$ be a separable indecomposable permutation of length at least $2$ which is of the form $q_1 \ominus ... \ominus q_k$ for some $k \geq 2$ where all of the $q_i$ are skew-indecomposable and $q_2, ..., q_{k-1}$ are not $1$. Then there are constants $k$ and $K$ such that
    \[
    kn^r|S_{n}(q)| \leq |S_{n,r}(q)| \leq Kn^r|S_{n}(q)|.
    \]
\end{theorem}

\begin{proof}
    We can apply Corollary $2.2$ immediately to get the lower bound. Since $q$ is separable, indecomposable, and of length at least $2$, then $q$ must be skew-decomposable. Hence we may also apply Corollary $3.4$ to get the upper bound.
\end{proof}

\section{Analytic Results}
To obtain the new results in Theorems $5.5, 5.8$, and $5.9$, we will combine our main result Theorem $4.1$ with some well-known previous results about certain permutation patterns.

This first lemma is provided in \cite{Bona} and the proof is due to Bostan.
\begin{lemma}
    Let $f(z) = \sum_{n\geq 0} f_n z^n$ be a power series with nonnegative real coefficients that is analytic at the origin. Assume that $c,C,\gamma$, and $\alpha$ exist so that $\alpha \leq -1$ is an integer and for all (sufficiently large) positive integers $n$ the chain of inequalities
    \[
    cn^\alpha \gamma^n \leq f_n \leq Cn^\alpha \gamma^n
    \]
    holds. Then $f(z)$ is not an algebraic power series.
\end{lemma}

\vspace{10pt}
This second lemma is well-known.
\begin{lemma}
    Let $f(z) = \sum_{n\geq 0} f_n z^n$ be a power series with nonnegative real coefficients that is analytic at the origin. Assume that $c,C,\gamma$, and $\alpha$ exist so that $\alpha$ is not a non-negative integer and for all (sufficiently large) positive integers $n$ the chain of inequalities
    \[
    cn^\alpha \gamma^n \leq f_n \leq Cn^\alpha \gamma^n
    \]
    holds. Then $f(z)$ is not a rational power series.
\end{lemma}

We also need the following known results.

\begin{theorem}[Regev 1981 \cite{Regev}]
    For all $k\geq 2$ there is $C_k$ such that
    \[
    |S_n(k(k-1)...1)| \sim C_k \frac{(k-1)^{2n}}{n^{(k^2-2k)/2}}.
    \]
\end{theorem}

\begin{theorem}[Backelin, West, Xin 2007 \cite{BackelinWestXin}]
Let $k \geq 1$. Let $q$ be a permutation. Then
\[
|S_n((1 \ 2 \ ... \ k) \ominus q)| = |S_n((k  \ k-1 \ ...  \ 1 )
\ominus q)|
\]
\end{theorem}

Combining all of these with Theorem $4.1$, we obtain the following result.
\begin{theorem}
    Suppose $r\geq 1$. Then there is $c_k$ and $C_k$ such that
    \[
    c_k \frac{(k-1)^{2n}}{n^{((k^2-2k)/2) - r}} \leq |S_{n,r}(23...k1)| \leq C_k \frac{(k-1)^{2n}}{n^{((k^2-2k)/2) - r}}.
    \]
    Let $S_r(z) = \sum_{n\geq 0} |S_{n,r}(23...k1)| z^n$. If $k$ is odd, then $S_r(z)$ is not rational. If $k$ is even and $r < \frac{k^2-2k}{2}$, then $S_r(z)$ is not algebraic.
\end{theorem}

The following results are also well-known:
\begin{theorem}
    For all $n\geq 1$ we have
    \[
    |S_n(3412)| = |S_n(4321)|.
    \]
\end{theorem}

\begin{theorem}[Bóna 1997 \cite{Bona1997b}]
    Let $S(z)$ be the ordinary generating function for $S_n(1342)$. Then
    \[
    S(z) = \frac{32z}{-8z^2 + 12z + 1 - (1-8z)^{3/2}}
    \]
    and
    \[
    S_n(1342) \sim C \frac{8^n}{n^{5/2}}
    \]
    for a (known) constant $C$.
\end{theorem}

Taking these together with Theorem $4.1$, we obtain the following results.

\begin{theorem}
    Suppose $r\geq 1$. Then there is $c$ and $C$ such that
    \[
    c \frac{9^{n}}{n^{4 - r}} \leq |S_{n,r}(3412)| \leq C \frac{9^{n}}{n^{4 - r}}.
    \]
    Let $S(z) = \sum_{n\geq 0} |S_{n,r}(3412)| z^n$. If $r < 4$, then $S(z)$ is not algebraic.
\end{theorem}

\begin{theorem}
    Suppose $r\geq 1$. Then there is $c$ and $C$ such that
    \[
    c \frac{8^{n}}{n^{5/2 - r}} \leq |S_{n,r}(4213)| \leq C \frac{8^{n}}{n^{5/2 - r}}.
    \]
    Let $S(z) = \sum_{n\geq 0} |S_{n,r}(4213)| z^n$. Then $S(z)$ is not rational.
\end{theorem}

\section{Future directions}

Say that $p$ and $p'$ are in the same \textit{effective Wilf class} if for all $n,r$ it is true that $S_{n,r}(p) = S_{n,r}(p')$. For patterns of length $4$ it is known that there are $7$ different effective Wilf classes, with representatives $4321, 4312, 4123, 3412, 4231, \\ 4213$, and $3142$.

The following conjecture appears in \cite{Conway} and is supported in that paper by substantial numerical evidence.

\begin{conjecture}[Conway and Guttmann 2023 \cite{Conway}]
     Let $q\in \{4321, 4312\}$. Then there is a constant $K$ so that
     \[
     S_{n,r}(q) \sim KS_{n}(q).
     \]
     Let $q\in \{4123, 3412, 4231, 4213, 3142\}$. Then there is a constant $K$ so that
     \[
     S_{n,r}(q) \sim Kn^rS_{n}(q).
     \]
\end{conjecture}

A result from our previous work \cite{Waite2025} supports this conjecture for $q = 4321$, and our new results support this conjecture for $q \in \{4123, 3412, 4231, 4213 \}$. A different argument which does not generalize very well can be used to obtain a similar result, supporting the conjecture for $q = 3142$. 

However, none of the techniques we have used have applied for $q = 4312$ and none of our attempts to obtain a similar result for this final case have worked. We hope to find an answer for this case in  future work, and perhaps extend this to another large set of patterns as in this paper.

\end{document}